\documentclass[12pt,a4paper]{amsart}
\usepackage{fullpage}
\usepackage{url}

\usepackage{algpseudocode,algorithm,algorithmicx}

\newtheorem{theorem}{Theorem}[section]

\theoremstyle{definition}
\newtheorem{remark}{Remark}[section]
\theoremstyle{definition}

\title{A remark on a construction of D.S.~Asche}
\author{Ferenc Sz\"oll\H{o}si}
\date{\today. Preprint. This research was supported in part by the Academy of Finland, Grant \#289002}
\address{F. Sz.: Department of Communications and Networking, Aalto University School of Electrical Engineering, P.O. Box 15400, 00076 Aalto, Finland}
\email{szoferi@gmail.com}

\begin{document}
\begin{abstract}
It is shown that a classical construction of D.S.~Asche of $72$ equiangular lines in the $19$ dimensional Euclidean space contains a subset of $54$ equiangular lines embedded in an $18$ dimensional subspace.
\end{abstract}
\maketitle

\section{Introduction and main results}
A set of $n$ lines, spanned by the unit vectors $v_1$, $\dots$, $v_n$ in a $d$-dimensional Euclidean space is called equiangular, if there is a common angle $\alpha\geq0$, such that for all $i\neq j$, $i,j\in\{1,\hdots,n\}$ we have $|\left\langle v_i, v_j\right\rangle|=\alpha$. The concept of equiangular lines was introduced in \cite{HAA}, and its theory was developed in \cite{LS}, \cite{vLS}. The maximum number of equiangular lines, denoted by $N(d)$, is known for $d\leq 13$, and in several sporadic dimensions, but for $d\geq 14$ the problem is wide open in general. For an up-to-date table with the most recent lower and upper bounds on $N(d)$, and for a list of additional references we refer the reader to \cite{GRWSN}, \cite{GRWSN2}.

The purpose of this note is to establish the following result.
\begin{theorem}\label{theorem1}
We have $N(18)\geq 54$.
\end{theorem}
\begin{proof}
Consider the extended binary Golay code as described in \cite[p.~719]{HANDBOOK}. The codewords are represented by the vectors generated by the rows of the $12\times 24$ matrix $\left[\begin{smallmatrix}I_6 & O & 0 & \textbf{1}\\ O & I_6 & \textbf{1}' & A\end{smallmatrix}\right]$, where $O$ is the $6\times 6$ zero matrix, $\textbf{1}$ and its transpose $\textbf{1}'$ are vectors of length $11$ having all entries equal to $1$, and $A$ is a $11\times 11$ circulant matrix with first row $[0,1,0,0,0,1,1,1,0,1,1]$. Let $\mathcal{C}$ denote the $759$-element subset of codewords of weight $8$. Let $e_i$ denote the standard basis vectors of $\mathbb{R}^{24}$ for $i\in\{1,\hdots,24\}$, and let $e_\Sigma:=\sum_{i=1}^{24}e_i$. For a codeword $x\in\mathcal{C}$ let $f(x):=(4x-4e_1-e_\Sigma)/\sqrt{80}$. Let $c_1:=e_2+e_3+e_{14}+e_{15}+e_{16}+e_{19}+e_{22}+e_{23}\in\mathcal{C}$, and let $c_2:=e_2+e_3+e_9+e_{11}+e_{12}+e_{13}+e_{21}+e_{24}\in\mathcal{C}$, and finally, let $m:=2e_4-e_5-e_6+2e_7-e_8-e_{10}+2e_{17}-e_{18}-e_{20}-3e_{22}+3e_{23}\in\mathbb{R}^{24}$. Then one readily verifies that the subset of all vectors $f(d)$, $d\in \mathcal{C}$, for which $\left\langle d,e_1\right\rangle=1$, and $\left\langle f(d),4e_1+e_\Sigma\right\rangle=\left\langle f(d),e_1-e_2\right\rangle=\left\langle f(d),e_1-e_3\right\rangle=\left\langle f(d),c_1\right\rangle=\left\langle f(d),c_2\right\rangle=\left\langle f(d),m\right\rangle=0$ forms an equiangular line system of $54=3\cdot 18$ lines in $\mathbb{R}^{18}$ with common angle $1/5$.
\end{proof}
The proof of Theorem~\ref{theorem1} is based on a construction of D.S.~Asche who exhibited $72$ equiangular lines in $\mathbb{R}^{19}$. Indeed, the only contribution of this paper is presenting the vector $m(c_1,c_2)$ which then removes $18$ out of these $72$ lines. Asche's construction is described in D.~Taylor's thesis \cite[p.~124]{TTH} (a more accessible reference is \cite[Example~5.19]{GRWSN2}), who remarked that $N(18)\geq 48$ follows. In an independent construction S.~Snover used the residual of a Steiner triple system on $19$ symbols to show the lower bound $N(17)\geq 48$. See \cite{LS} regarding the historical remarks, and \cite[Example~5.18]{GRWSN2} for an explicit construction.
\begin{remark}
We offer the following geometric interpretation of the construction described in the Proof of Theorem~1.1. From the $72$ equiangular lines (represented by unit vectors) coming from Asche's construction, we remove two $9$-subset of unit vectors which form two disjoint cliques aligned in the `right' way. More precisely, we remove a subset of $18$ vectors $u_i$ and $v_i$, $i\in\{1,\dots,9\}$ for which the following conditions hold: $\left\langle u_i,u_j\right\rangle=\left\langle v_i,v_j\right\rangle = 1/5$, $i\neq j$, $i,j\in\{1,\dots,9\}$ and in addition for every vector $u_i$, $i\in\{1,\dots,9\}$ there are exactly two distinct vectors $v_j\neq v_k$, $j,k\in\{1,\dots,9\}$ such that $\left\langle u_i,v_j\right\rangle=\left\langle u_i,v_k\right\rangle=1/5$ (and the other way around). We observe that the two cliques are on either side of the $1$-codimensional hyperplane orthogonal to $m$, as $\left\langle u_i,m\right\rangle=6/\sqrt{5}$ whereas $\left\langle v_i,m\right\rangle =-6/\sqrt{5}$ for every $i\in\{1,\dots,9\}$.
\end{remark}

We conclude this note with three further observations.

First, the configuration given in the Proof of Theorem~\ref{theorem1} is not extendible to a larger set of equiangular lines in $\mathbb{R}^{18}$.

Secondly, the Seidel matrix $S$ corresponding to this newly constructed equiangular line system (with $[S]_{i,j}:=5\left(\left\langle v_i,v_j\right\rangle-\delta_{ij}\right)$, where $i,j\in\{1,\dots, 54\}$, see \cite{SZO}) has spectrum
\begin{equation*}
\Lambda(S)=\{[-5]^{36}, [7]^6, [11]^8, [13]^2, [12-\sqrt{37}]^1, [12+\sqrt{37}]^1\}
\end{equation*}
(exponents denote multiplicities), and automorphism group size of $|\mathrm{Aut}(S)|=216=4\cdot 54$.

Finally, amongst those sub-Seidel matrices of $S$, which are of order $n\in\{50,\dots,53\}$ there is a unique one (up to equivalence) with integral spectrum. The spectrum of this Seidel matrix $T$ of order $52$ reads: $\Lambda(T)=\{[-5]^{34}, [3]^1, [5]^1, [7]^6, [11]^7, [13]^2, [17]^1\}$. 

\end{document}